\documentclass[12pt,leqno]{article}
\usepackage[shortlabels]{enumitem}
\usepackage{preamble}

\title{Positively curved Finsler metrics on vector bundles II}
\author{Kuang-Ru Wu}
\newcommand{\RN}[1]{%
  \textup{\uppercase\expandafter{\romannumeral#1}}%
}

\usepackage{fullpage}
\usepackage{tikz-cd}

\usepackage{dutchcal}

\usepackage{fancyhdr}

\theoremstyle{plain}

\numberwithin{equation}{section}

\begin{document}

\date{}

\parskip=6pt

\maketitle
\begin{abstract}
We show that if $E$ is an ample vector bundle of rank at least two with some curvature bound on  $O_{P(E^*)}(1)$, then $E^*\otimes \det E$ is Kobayashi positive. The proof relies on comparing the curvature of $(\det E^*)^k$ and $S^kE$ for large $k$ and using duality of convex Finsler metrics. Following the same thread of thought, we show if $E$ is ample with similar curvature bounds on  $O_{P(E^*)}(1)$ and $O_{P(E\otimes \det E^*)}(1)$, then $E$ is Kobayashi positive. With additional assumptions, we can furthermore show that $E^*\otimes \det E$ and $E$ are Griffiths positive.
\end{abstract}

\section{Introduction}
Let $E$ be a holomorphic vector bundle of rank $r$ over a compact complex manifold $X$ of dimension $n$. We denote the dual bundle by $E^*$ and its projectivized bundle by $P(E^*)$. The vector bundle $E$ is said to be ample if the line bundle $O_{P(E^*)}(1)$ over $P(E^*)$ is ample. On the other hand, $E$ is called Griffiths positive if $E$ carries a Griffiths positive Hermitian metric. Moreover, $E$ is called Kobayashi positive if $E$ carries a strongly pseudoconvex Finsler metric whose Kobayashi curvature is positive (see \cite[Section 2]{wu_2022} for a quick review).

There are two conjectures made by Griffiths \cite{Griff69} and Kobayashi \cite{Negfinsler} regarding the equivalence of ampleness and positivity:
\begin{align*}
    &(1) \text{ If $E$ is ample, then $E$ is Griffiths positive.  }\\
    &(2) \text{ If $E$ is ample, then $E$ is Kobayashi positive.  }
\end{align*}
Not much is known about these two conjectures save for $n=1$ by \cite{Umemura,CampanaFlenner} (see recent progress \cite{Berndtsson09,MourouganeTaka,toric,positivityandvanishingthmliu,LY15,naumann2017approach,FengLiuWan, demailly2020hermitianyangmills, finski2020monge, pingali2021note}). Note that the converse of each conjecture is true (\cite{FengLiuWan,wu_2022}).

By Kodaira's embedding theorem, ampleness of a line bundle is equivalent to the existence of a positively curved metric on the line bundle. So, the conjectures of Griffiths and Kobayashi can be rephrased as: given a positively curved metric on $O_{P(E^*)}(1)$, can we construct a positively curved Hermitian/Finsler metric on $E$? In this paper, we show that it is so, by imposing curvature bounds on tautological line bundles of $P(E^*)$ and $P(E)$. Since Hermitian metrics on $O_{P(E^*)}(1)$ are in one-to-one correspondence with Finsler metrics on $E^*$, these curvature bounds can also be written in terms of Kobayashi curvature.

We first consider a relevant case where the picture is clearer. It is known that, for rank of $E$ at least 2, 
\begin{align*}
    &(1) \text{ If $E$ is Griffiths positive then $E^*\otimes \det E$ with the induced metric is Griffiths positive.}\\
    &(2) \text{ If $E$ is ample then $E^*\otimes \det E$ is ample. }
\end{align*}
The first fact can be found in \cite[P. 346, Theorem 9.2]{demailly1997complex}, and the second in \cite[Corollary 5.3]{Hart66}. If we follow the guidance of Griffiths and Kobayashi, we would ask whether ampleness of $E$ implies Griffiths/Kobayashi positivity of $E^*\otimes \det E$ for $r \geq 2$. Our first result is that this can be achieved by imposing curvature bounds on $O_{P(E^*)}(1)$. 

Let $q:P(E^*)\to X$ be the projection. Let $g$ be a metric on $O_{P(E^*)}(1)$ whose curvature restricted to a fiber $\Theta(g)|_{P(E^*_z)}$ is positive for all $z\in X$. For a tangent vector $\eta\in T^{1,0}_{z}X$ and a point $[\zeta]\in P(E^*_z)$, we consider tangent vectors $\tilde{\eta}$ to $P(E^*)$ at $(z,[\zeta])$ such that $q_*(\tilde{\eta})=\eta$, namely the lifts of $\eta$ to $T^{1,0}_{(z,[\zeta])}P(E^*)$. Then we define the function
\begin{equation}\label{inf}
     (\eta,[\zeta])\mapsto \inf_{q_*(\tilde{\eta})=\eta} \Theta(g)(\tilde{\eta},\bar{\tilde{\eta}})
\end{equation}
the inf taken over all the lifts of $\eta$ to $T^{1,0}_{(z,[\zeta])}P(E^*)$. This inf is actually a minimum, see (\ref{min}). On the other hand, since such a metric $g$ corresponds to a strongly pseudoconvex Finsler metric on $E^*$, and if we denote its Kobayashi curvature by $\theta(g)$ a $(1,1)$-form on $P(E^*)$, then 
\begin{equation}\label{inf=kob}
    \inf_{q_*(\tilde{\eta})=\eta} \Theta(g)(\tilde{\eta},\bar{\tilde{\eta}})=-\theta(g)(\tilde{\eta},\bar{\tilde{\eta}}).
\end{equation}
The term on the right is independent of the choice of lifts $\tilde{\eta}$ (we will review Finsler metrics and prove (\ref{inf=kob}) in Subsection \ref{system}).

\begin{theorem}\label{thm 1}
Assume $r\geq 2$ and the line bundle $O_{P(E^*)}(1)$ has a positively curved metric $h$ and a metric $g$ with $\Theta(g)|_{P(E^*_z)}>0$ for all $z\in X$. If there exist a Hermitian metric $\Omega$ on $X$ and a constant $M\in [1,r)$ such that the following inequalities of $(1,1)$-forms hold
\begin{align}
    M q^*\Omega &\geq -\theta(g)\text{ and }\label{thm 1 assumption}\\
     q^*\Omega &\leq -\theta(h) ,\label{thm 1 assumption 2}
\end{align}
then $E^*\otimes \det E$ is Kobayashi positive. 
\end{theorem}

We can of course choose $g$ to be $h$ in Theorem \ref{thm 1}, but it does not seem to help too much except simplifying the statement. The proof of Theorem \ref{thm 1} relies on two observations. First, starting with $g$ and $h$ on $O_{P(E^*)}(1)$, we construct two Hermitian metrics on $S^kE$ and $\det E$ respectively. The curvature of the induced metric on $S^kE\otimes (\det E^*)^k$ can be shown to be Griffiths negative for $k $ large (see Section \ref{sec pf of thm1} for details). The second observation which we use in \cite{wu_2022} already is that since the induced metric on $S^kE\otimes (\det E^*)^k$ is basically an $L^2$-metric, its $k$-th root is a convex Finsler metric on $E\otimes \det E^*$ which is also strongly plurisubharmonic on the total space minus the zero section. After perturbing this Finsler metric and taking duality, we get a convex and strongly pseudoconvex Finsler metric on $E^*\otimes \det E$ whose Kobayashi curvature is positive. So the bundle $E^*\otimes \det E$ is Kobayashi positive. Notice that the Finsler metric we find is actually convex.

Now let us go back to the original conjecture of Kobayashi and adapt the proof of Theorem \ref{thm 1} to this case. Let $p:P(E)\to X$ be the projection. We recall under the canonical isomorphism $P(E\otimes \det E^*)\simeq P(E)$, the line bundle $O_{P(E\otimes \det E^*)}(1)$ corresponds to the line bundle $O_{P(E)}(1)\otimes p^*\det E$ (see \cite[Page 86, Prop. 3.6.21]{MR909698}). Let $g$ be a metric on $O_{P(E)}(1)\otimes p^*\det E$ with $\Theta(g)|_{P(E_z)}>0$ for all $z\in X$. For a tangent vector $\eta\in T^{1,0}_{z}X$ and a point $ [\xi]\in P(E_{z})$, we similarly have $$ (\eta,[\xi])\mapsto \inf_{p_*(\eta')=\eta} \Theta(g)(\eta',\bar{\eta}');$$ here $\eta'$ are the lifts of $\eta$ to $T^{1,0}_{(z,[\xi])}P(E)$. Meanwhile, such a metric $g$ corresponds to a strongly pseudoconvex Finsler metric on $E\otimes \det E^*$, and we denote its Kobayashi curvature by $\theta(g)$ a $(1,1)$-form on $P(E)$. As before, 
\begin{equation}
    \inf_{p_*(\eta')=\eta} \Theta(g)(\eta',\bar{\eta}')=-\theta(g)(\eta',\bar{\eta}').
\end{equation}

\begin{theorem}\label{thm 2}
Assume $r\geq 2$ and $O_{P(E^*)}(1)$ has a positively curved metric $h$ and $O_{P(E)}(1)\otimes p^*\det E$ has a metric $g$ with $\Theta(g)|_{P(E_z)}>0$ for all $z\in X$. If there exist a Hermitian metric $\Omega$ on $X$ and a constant $M\in [1,r)$ such that
\begin{align}\label{thm 2 assumption}
    M p^*\Omega &\geq -\theta(g) \text{ and }\\   q^* \Omega &\leq -\theta(h),
\end{align}
then $E$ is Kobayashi positive.
\end{theorem}

Since ampleness of $E$ implies ampleness of $E^*\otimes \det E$, one choice for $g$ in Theorem \ref{thm 2} is a positively curved metric on $O_{P(E)}(1)\otimes p^*\det E$, but how much this choice helps is unknown to us. The proof of Theorem \ref{thm 2} follows the same scheme as in Theorem \ref{thm 1}. We first use $h$ and $g$ to construct Hermitian metrics on $\det E$ and $S^kE^* \otimes (\det E)^k$ respectively. The induced metric on  $[S^kE^* \otimes (\det E)^k]\otimes (\det E^*)^k$ is Griffiths negative for $k$ large (see Section \ref{sec pf o thm2}). Then by taking $k$-th root, perturbing, and taking duality, we obtain a convex, strongly pseudoconvex, and Kobayashi positive Finsler metric on $E$.

The conclusions in Theorems \ref{thm 1} and \ref{thm 2} are about Finsler metrics. For their Hermitian counterpart, we need additional assumptions. The reason is that in Theorems \ref{thm 1} and \ref{thm 2}, taking large tensor power of various bundles helps us eliminate the curvature of the relative canonical bundles $K_{P(E^*)/X}$ and $K_{P(E)/X}$, and after getting the desired estimates we take $k$-th root to produce Finsler metrics. However, the step of taking $k$-th root produces only Finsler, not Hermitian metrics. So the first step of taking large tensor power is not allowed if one wants Hermitian metrics.

Let us be more precise. For a metric $g$ on $O_{P(E^*)}(1)$ with $\Theta(g)|_{P(E^*_z)}>0$ for all $z\in X$, we denote $\Theta(g)|_{P(E^*_z)}$ by $\omega_z$ for the moment. The relative canonical bundle $K_{P(E^*)/X}$ has a metric induced from $\{\omega_z^{r-1} \}_{z\in X}$, and we denote the corresponding curvature by $\gamma_g$, a $(1,1)$-form on $P(E^*)$. For $\eta\in T^{1,0}_z X$ and $[\zeta]\in P(E^*_z)$, we consider $$(\eta,[\zeta])\mapsto \sup_{q_*(\tilde{\eta})=\eta} \gamma_g(\tilde{\eta},\bar{\tilde{\eta}})$$ the sup taken over all the lifts of $\eta $ to $T^{1,0}_{(z,[\zeta])}P(E^*)$. The sup is a maximum under a suitable assumption, see (\ref{min'}). Moreover, for $z\in X$, the restriction $\gamma_g|_{P(E^*_z)}$ is actually the negative of Ricci curvature $-\Ric_{\omega_z}$ of the metric $\omega_z$ on $P(E^*_z)$.

Any Hermitian metric $G$ on $E^*$ will induce a metric $g$ on $O_{P(E^*)}(1)$ with $\Theta(g)|_{P(E^*_z)}>0$ and $\gamma_g|_{P(E^*_z)}<0$ for all $z\in X$. Indeed, in this case, $\Theta(g)|_{P(E^*_z)}$ is the Fubini--Study metric and its Ricci curvature is positive, so $\gamma_g|_{P(E^*_z)}<0$. Furthermore, for any $\eta\in T^{1,0}_zX$ and any $[\zeta]\in P(E^*_z)$,
\begin{equation}\label{sup}
    \sup_{q_*(\tilde{\eta})=\eta} \gamma_g(\tilde{\eta},\bar{\tilde{\eta}})=r\theta(g)(\tilde{\eta},\bar{\tilde{\eta}})-q^*\Theta(\det G)(\tilde{\eta},\bar{\tilde{\eta}})
\end{equation}
(We will prove (\ref{sup}) in Subsection \ref{sub 2.5}).

\begin{theorem}\label{thm 3}
Assume $r\geq 2$ and the line bundle $O_{P(E^*)}(1)$ has a positively curved metric $h$ and a metric $g$ induced from a Hermitian metric $G$ on $E^*$. If there exist a Hermitian metric $\Omega$ on $X$ and a constant $M\in [1,r)$ such that
\begin{align}\label{thm 3 assumption}
    M q^*\Omega &\geq -(r+1)\theta(g)+q^*\Theta(\det G)\text{ and}\\
    q^*\Omega &\leq-\theta(h), 
\end{align}
then $E^*\otimes \det E$ is Griffiths positive. 
\end{theorem}

Theorem \ref{thm 3} could be seen as a Hermitian analogue of Theorem \ref{thm 1}. To state a Hermitian analogue of Theorem \ref{thm 2}, we use again the isomorphism between $O_{P(E\otimes \det E^*)}(1)\to P(E\otimes \det E^*)$ and $O_{P(E)}(1)\otimes p^*\det E\to P(E)$. 


\begin{theorem}\label{thm 4}
Assume $r\geq 2$ and $O_{P(E^*)}(1)$ has a positively curved metric $h$, and $O_{P(E)}(1)\otimes p^*\det E$ has a metric $g$ induced from a Hermitian metric $G$ on $E\otimes \det E^*$. If there exist a Hermitian metric $\Omega$ on $X$ and a constant $M\in [1,r)$ such that 
\begin{align}\label{thm 4 assumption}
    M p^*\Omega &\geq -(r+1)\theta(g)+p^*\Theta(\det G)  \text{ and }\\   q^*\Omega &\leq -\theta(h),
\end{align}
then $E$ is Griffiths positive.
\end{theorem}

In all the theorems above, the existence of the metric $h$ comes from ampleness of $E$, and the existence of metrics $g$ can be deduced easily as explained. So the real assumptions lie in $\Omega$, $M$, and the inequalities they have to satisfy. To weaken or remove these inequalities, one possible direction is to use geometric flows as in  \cite{naumann2017approach,wan2018positivity,yury,li2021hermitian}. Another possible direction is to use the interplay between the optimal $L^2$-estimates and the positivity of curvature (see \cite{GuanZhou,BoLem,LLextrapolation,Hacon,ZhouZhu}).  

One example where the assumptions of all the theorems above are satisfied is given by $E=L^9\bigoplus L^8 \bigoplus L^7$ with $L$ a positive line bundle. The triple $(9,8,7)$ or the rank $r=3$ is not that important; the point is to make sure the eigenvalues of the curvature with respect to some positive $(1,1)$-form do not spread out too far. A more sophisticated example, related to approximate Hermitian--Yang--Mills metrics (\cite{Jacob,MisraRay,li2021hermitian}), is semistable ample vector bundles over Riemann surfaces (see Section \ref{example} for details of the examples).

The proof of Theorem \ref{thm 1} is given in Section \ref{sec pf of thm1}, and almost as a corollary we prove Theorem \ref{thm 2} in Section \ref{sec pf o thm2}. The proof of Theorem \ref{thm 3} in Section \ref{sec pf of thm3} is a modification of Theorem \ref{thm 1}, but we still write out the details. In Section \ref{sec pf of thm 4}, we prove Theorem \ref{thm 4}
based on Section \ref{sec pf of thm3}.

I am grateful to L\'aszl\'o Lempert for his comments that help improve this paper. I would like to thank Academia Sinica, National Center for Theoretical Sciences, and National Taiwan University for their support.

\section{Preliminaries}

\subsection{Finsler metrics}\label{system}
We will use some facts about Finsler metrics on vector bundles which can be found in \cite{Negfinsler,ComplexFinsler,caowong,AikouMSRI,wu_2022}. First, we recall the definition of Finsler metrics. Let $E^*$ be a holomorphic vector bundle of rank $r$ over a compact complex manifold $X$. For a vector $\zeta\in E^*_z$, we symbolically write $(z,\zeta)\in E^*$. A smooth Finsler metric $G$ on the vector bundle $E^*\to X$ is a real-valued function on $E^*$ such that \begin{align*}
    &(1) \text{ $G$ is smooth away from the zero section of $E^*$}.\\
    &(2) \text{ For $(z,\zeta)\in E^*$}, G(z,\zeta)\geq 0, \text{ and equality holds if and only if $\zeta=0$}.\\
    &(3) \text{ $G(z,\lambda\zeta)=|\lambda|^2G(z,\zeta)$,  for $\lambda\in \mathbb{C}$}.
\end{align*}

Let $g$ be a Hermitian metric on $O_{P(E^*)}(1)$ with $\Theta(g)|_{P(E^*_z)}>0$ for all $z\in X$. Such a $g$ corresponds to a strongly pseudoconvex Finsler metric $G$ on $E^*$, and there is a natural Hermitian metric $\tilde{G}$ on the pull-back bundle $q^*E^*$ where $q:P(E^*)\to X$ is the projection (see \cite[Section 2.2]{wu_2022}). Now $(q^*E^*,\tilde{G})$ is a Hermitian holomorphic vector bundle, so we can talk about its Chern curvature $\Theta$, an $\End q^*E^*$-valued $(1,1)$-form on $P(E^*)$. With respect to the metric $\tilde{G}$, the bundle $q^*E^*$ has a fiberwise orthogonal decomposition $O_{P(E^*)}(-1)\oplus O_{P(E^*)}(-1)^\perp$, and so $\Theta$ can be written as a block matrix. Let $\Theta|_{O_{P(E^*)}(-1)}$ denote the block in the matrix $\Theta$ corresponding to $\End(O_{P(E^*)}(-1))$. Since $O_{P(E^*)}(-1)$ is a line bundle, $\Theta|_{O_{P(E^*)}(-1)}$ is a $(1,1)$-form on $P(E^*)$, and it is called the Kobayashi curvature of the Finsler metric $G$. We will use $\theta(g)$ to denote the Kobayashi curvature 
\begin{equation}\label{def of koba}
    \theta(g):=\Theta|_{O_{P(E^*)}(-1)}.
\end{equation}

In order to relate the Kobayashi curvature $\theta(g)$ to the curvature $\Theta(g)$ of $g$, we consider coordinates normal at one point. Given a point $(z_0,[\zeta_0])\in P(E^*)$, there exists a holomorphic frame $\{s_i\}$ for $E^*$ around $z_0\in X$ such that 
\begin{equation}\label{normal}
 G_{\zeta_i\bar{\zeta}_j}(z_0,\zeta_0)=\delta_{ij},\,\,
 G_{\zeta_i\bar{\zeta}_j z_\alpha}(z_0,\zeta_0)=G_{\zeta_i\bar{\zeta}_j \bar{z}_\beta}(z_0,\zeta_0)=G_{\bar{\zeta}_j z_\alpha}(z_0,\zeta_0)=G_{z_\alpha}(z_0,\zeta_0)=0,
\end{equation}
where we use $\{\zeta_i\}$ for the fiber coordinates on $E^*$ with respect to the frame $\{s_i\}$, and $\{z_\alpha\}$ for the local coordinates on $X$ (such a frame can be obtained by (5.11) in \cite{ComplexFinsler}). Moreover if $\Omega$ is a Hermitian metric on $X$, then by a linear transformation in the $z$-coordinates, we can make $\Omega(\partial/\partial z_\alpha,\partial/\partial \bar{z}_\beta )(z_0)=\delta_{\alpha\beta}$ without affecting (\ref{normal}). We will call this coordinate system normal at the point $(z_0,[\zeta_0])\in P(E^*)$.

Around the point $(z_0,[\zeta_0])\in P(E^*)$, we assume the local coordinates $(z_1,\cdots, z_n, w_1,\cdots, w_{r-1})$ are given by $w_i=\zeta_i/\zeta_r$ for $i=1\sim r-1$. So $$e:=\frac{\zeta_1s_1+\cdots+\zeta_rs_r}{\zeta_r}=w_1s_1+\cdots+w_{r-1}s_{r-1}+s_r$$
is a holomorphic frame for $O_{P(E^*)}(-1)$. Let $e^*$ be the dual frame of $O_{P(E^*)}(1)$ around $(z_0,[\zeta_0])\in P(E^*)$, and $g(e^*,e^*)=e^{-\phi}$. Then, the curvature $\Theta(g)$ can be written locally as 
\begin{equation*}
    \sum_{\alpha,\beta}\frac{\partial^2  \phi  }{\partial z_\alpha \partial \bar{z}_\beta}dz_\alpha\wedge d\bar{z}_\beta+\sum_{\alpha,j}\frac{\partial^2  \phi  }{\partial z_\alpha \partial \bar{w}_j}dz_\alpha\wedge d\bar{w}_j+\sum_{i,\beta}\frac{\partial^2  \phi  }{\partial w_i \partial\bar{z}_\beta}dw_i\wedge d\bar{z}_\beta+\sum_{i,j}\frac{\partial^2  \phi  }{\partial w_i \partial\bar{w}_j}dw_i\wedge d\bar{w}_j.
\end{equation*}
Note that the terms $\partial^2  \phi/\partial z_\alpha \partial \bar{w}_j:=\phi_{\alpha\bar{j}}$ vanish at $(z_0,[\zeta_0])$ by (\ref{normal}) and the fact $e^\phi=1/g(e^*,e^*)=G(w_1s_1+\cdots+w_{r-1}s_{r-1}+s_r) $. For a tangent vector $\eta\in T^{1,0}_{z_0}X$, we can write $\eta=\sum_\alpha \eta_\alpha \partial /\partial z_\alpha$. For the lifts $\tilde{\eta}$ of $\eta$ to $T^{1,0}_{(z_0,[\zeta_0])}P(E^*)$, we have 
\begin{equation}\label{min}
    \inf_{q_*(\tilde{\eta})=\eta} \Theta(g)(\tilde{\eta},\bar{\tilde{\eta}})=\sum_{\alpha,\beta} \phi_{\alpha\bar{\beta}}|_{(z_0,[\zeta_0])}\eta_\alpha \bar{\eta}_\beta
\end{equation}
because $\phi_{\alpha\bar{j}}=0$ at $(z_0,[\zeta_0])$ and the matrix $(\phi_{i\bar{j}})$ is positive. On the other hand, using the same coordinate system, the curvature $\Theta$ of $\tilde{G}$
can be written as
\begin{equation*}
    \Theta=\sum_{\alpha,\beta}R_{\alpha\bar{\beta}}\,dz_\alpha\wedge d\bar{z}_\beta+\sum_{\alpha,l}P_{\alpha\bar{l}}\,dz_\alpha\wedge d\bar{w}_l+\sum_{k,\beta}\mathcal{P}_{k\bar{\beta}}\,dw_k\wedge d\bar{z}_\beta+\sum_{k,l}Q_{k\bar{l}}\,dw_k\wedge d\bar{w}_l,
\end{equation*}
where $R_{\alpha\bar{\beta}},P_{\alpha\bar{l}},\mathcal{P}_{k\bar{\beta}}$, and $Q_{k\bar{l}}$ are endomorphisms of $q^*E^*$. By \cite[Formula (2.4)]{wu_2022}, for any lift $\tilde{\eta}$ of $\eta$ to $T^{1,0}_{(z_0,[\zeta_0])}P(E^*)$, we have
\begin{equation}\label{wha}
  \theta(g)(\tilde{\eta},\bar{\tilde{\eta}})=\Theta|_{O_{P(E^*)}(-1)}(\tilde{\eta},\bar{\tilde{\eta}})=\sum_{\alpha,\beta} \frac{\tilde{G}(R_{\alpha\bar{\beta}}\zeta_0,\zeta_0)}{\tilde{G}(\zeta_0,\zeta_0)}\eta_\alpha \bar{\eta}_\beta=-\sum_{\alpha,\beta} \phi_{\alpha\bar{\beta}}|_{(z_0,[\zeta_0])}\eta_\alpha \bar{\eta}_\beta,
\end{equation}
where the last equality is by \cite[Formula 5.16]{ComplexFinsler}.

From (\ref{min}) and (\ref{wha}), we see
$$\inf_{q_*(\tilde{\eta})=\eta} \Theta(g)(\tilde{\eta},\bar{\tilde{\eta}})=-\theta(g)(\tilde{\eta},\bar{\tilde{\eta}})$$ 
which is formula (\ref{inf=kob}) we claim in the Introduction, and when evaluated using normal coordinates they are 
$\sum_{\alpha,\beta} \phi_{\alpha\bar{\beta}}|_{(z_0,[\zeta_0])}\eta_\alpha \bar{\eta}_\beta$.

\subsection{Hermitian metrics}\label{sub 2.5}
This subsection is a special case of Subsection \ref{system}, and it will be used in the proofs of Theorems \ref{thm 3} and \ref{thm 4}. Let $G$ be a Hermitian metric on the bundle $E^*$. The pull-back bundle $q^*E^*\to P(E^*)$ with the pull-back metric $q^*G$ induces a metric $g^*$ on the subbundle $O_{P(E^*)}(-1)$. We denote the dual metric on $O_{P(E^*)}(1)$ by $g$.

Let $\Omega$ be a Hermitian metric on $X$ and $z_0$ a point in $X$ with local coordinates $\{z_\alpha \}$ such that $\Omega(\partial/\partial z_\alpha, \partial/ \partial \bar{z}_\beta)(z_0)=\delta_{\alpha\beta}$. There exists a holomorphic frame $\{s_i\}$ for $E^*$ around $z_0$ such that $G(s_i,s_j)=\delta_{ij}+O(|z|^2)$ where $z_0$ corresponds to the origin in the local coordinates. We use $\{\zeta_i\}$ for the fiber coordinates with respect to the frame $\{s_i\}$. For a point $(z_0,[\zeta_0])\in P(E^*)$, we assume the local coordinates $(z_1,\cdots, z_n, w_1,\cdots, w_{r-1})$ around $(z_0,[\zeta_0])$ are given by $w_i=\zeta_i/\zeta_r$ for $i=1\sim r-1$. So $$e:=\frac{\zeta_1s_1+\cdots+\zeta_rs_r}{\zeta_r}=w_1s_1+\cdots+w_{r-1}s_{r-1}+s_r$$
is a holomorphic frame for $O_{P(E^*)}(-1)$, and
\begin{align*}
  g^*(e,e)&=q^*G(w_1s_1+\cdots+w_{r-1}s_{r-1}+s_r,w_1s_1+\cdots+w_{r-1}s_{r-1}+s_r)\\
  &=1+O(|z|^2)+O(|w|^2)+O(|w||z|^2)+O(|w|^2|z|^2).  
\end{align*}
The $z_\alpha$-derivative of $g^*(e,e)$ is $g^*(e,e)_{z_\alpha}=O((1+|w|+|w|^2)|z|)$, hence the $w_i$-derivatives of $g^*(e,e)_{z_\alpha}$ of any order are zero when evaluated at $z_0$. Therefore, if we denote $g^*(e,e)$ by $e^{\phi}$, then at $z_0$ 
\begin{equation}\label{vanish}
    \phi_{\alpha\bar{j}}=\phi_{\alpha i \bar{j}}=\phi_{\alpha i \bar{j}\bar{k}}=0 \text{, and } \big(\log \det (\phi_{i\bar{j}}) \big)_{\alpha\bar{k}}=0.
\end{equation}
In this coordinate system, the curvature $\Theta(g)$ is 
\begin{equation*}
    \sum_{\alpha,\beta}\frac{\partial^2  \phi  }{\partial z_\alpha \partial \bar{z}_\beta}dz_\alpha\wedge d\bar{z}_\beta+\sum_{\alpha,j}\frac{\partial^2  \phi  }{\partial z_\alpha \partial \bar{w}_j}dz_\alpha\wedge d\bar{w}_j+\sum_{i,\beta}\frac{\partial^2  \phi  }{\partial w_i \partial\bar{z}_\beta}dw_i\wedge d\bar{z}_\beta+\sum_{i,j}\frac{\partial^2  \phi  }{\partial w_i \partial\bar{w}_j}dw_i\wedge d\bar{w}_j.
\end{equation*}
For a tangent vector $\eta\in T^{1,0}_{z_0}X$, we can write $\eta=\sum_\alpha \eta_\alpha \partial /\partial z_\alpha$. For the lifts $\tilde{\eta}$ of $\eta$ to $T^{1,0}_{(z_0,[\zeta_0])}P(E^*)$, we have 
\begin{equation}\label{min 2}
    \inf_{q_*(\tilde{\eta})=\eta} \Theta(g)(\tilde{\eta},\bar{\tilde{\eta}})=\sum_{\alpha,\beta} \phi_{\alpha\bar{\beta}}|_{(z_0,[\zeta_0])}\eta_\alpha \bar{\eta}_\beta
\end{equation}
because $\phi_{\alpha\bar{j}}=0$ at $z_0$ and the matrix $(\phi_{i\bar{j}})$ is positive. Since $G$ is a Hermitian metric, the corresponding Kobayashi curvature is  
\begin{equation}\label{Herm}
 \theta(g)=q^*\Theta(G)|_{O_{P(E^*)}(-1)},   
\end{equation}
which is equal to the negative of (\ref{min 2}) by Subsection \ref{system}.

Using the same coordinate system, the restriction $\Theta(g)|_{P(E^*_z)}$ is $\sum \phi_{i\bar{j}} dw_i \wedge d\bar{w}_j$, so the metric on $K_{P(E^*)/X}$ induced from $\{(\Theta(g)|_{P(E^*_z)})^{r-1}\}_{z\in X}$ has its curvature $\gamma_g$ equal to 
\begin{equation}\label{gamma}
\begin{aligned}
    &\sum_{\alpha,\beta} (\log \det (\phi_{i\bar{j}}))_{\alpha\bar{\beta}}    dz_\alpha\wedge d\bar{z}_\beta+\sum_{\alpha,j}    (\log \det (\phi_{i\bar{j}}))_{\alpha\bar{j}}    dz_\alpha\wedge d\bar{w}_j\\
    +&\sum_{i,\beta}    (\log \det (\phi_{i\bar{j}}))_{i\bar{\beta}}    dw_i\wedge d\bar{z}_\beta+\sum_{i,j}    (\log \det (\phi_{i\bar{j}}))_{i\bar{j}} dw_i\wedge d\bar{w}_j.
\end{aligned}
\end{equation}
The matrix 
$\big((\log \det (\phi_{i\bar{j}}))_{i\bar{j}}\big)$ is negative because it represents the negative of the Ricci curvature of the Fubini--Study metric on $P(E^*_z)$. Moreover, the terms $(\log \det (\phi_{i\bar{j}}) )_{\alpha\bar{j}}=0$ at $z_0$ by (\ref{vanish}). As a result, for a tangent vector $\eta\in T^{1,0}_{z_0}X$ with $\eta=\sum \eta_\alpha \partial/\partial  z_\alpha$ in this coordinate system, we have 
\begin{equation}\label{min'}
\sup_{q_*(\tilde{\eta})=\eta} \gamma_g (\tilde{\eta},\bar{\tilde{\eta}})= \sum_{\alpha,\beta} (\log \det (\phi_{i\bar{j}}))_{\alpha\bar{\beta}}|_{(z_0,[\zeta_0])} \eta_\alpha \bar{\eta}_\beta, \end{equation}
where $\tilde{\eta}$ are the lifts of $\eta$ to $T^{1,0}_{(z_0,[\zeta_0])}P(E^*)$.    

Finally, the metric on $K_{P(E^*)/X}$ induced from $\{(\Theta(g)|_{P(E^*_z)})^{r-1}\}_{z\in X}$ can be identified with the metric $(g^*)^r\otimes q^* (\det G^*)$ under the isomorphism $K_{P(E^*)/X}\simeq O_{P(E^*)}(-r)\otimes q^* \det E$ (\cite[Page 85, Prop. 3.6.20]{MR909698}). This fact can be verified at one point using the normal coordinates above. Therefore, \begin{equation}\label{gamma 1}
 \gamma_g=-r\Theta(g)-q^*\Theta(\det G). 
\end{equation}
So, for any $\eta\in T^{1,0}_zX$ and any $[\zeta]\in P(E^*_z)$,
\begin{equation*}
    \sup_{q_*(\tilde{\eta})=\eta} \gamma_g(\tilde{\eta},\bar{\tilde{\eta}})=-r\inf_{q_*(\tilde{\eta})=\eta} \Theta(g)(\tilde{\eta},\bar{\tilde{\eta}})-\Theta(\det G)(\eta,\bar{\eta})=r\theta(g)(\tilde{\eta},\bar{\tilde{\eta}})-\Theta(\det G)(\eta,\bar{\eta}).
\end{equation*}
This is formula (\ref{sup}) that we promise to prove in the Introduction.

\subsection{Convexity}

Let $E$ be a holomorphic vector bundle of rank $r$ over a compact complex manifold $X$. Given a Hermitian metric $H_k$ on the symmetric power $S^kE$, we can define a Finsler metric on $E$ by assigning to $u\in E$ length $H_k(u^k,u^k)^{1/2k}$. We will denote this Finsler metric by $H_k^{1/2k}$, namely $H_k^{1/2k}(u)=H_k(u^k,u^k)^{1/2k}$.
\begin{lemma}\label{convex}
Let $F_1$ be a vector bundle and $F_2$ a line bundle over $X$. Assume $F_2$ carries a Hermitian metric $H$. We also assume, for some $k$, $S^kF_1$ carries a Hermitian metric $H_k$ such that the induced Finsler metric $H_k^{1/2k}$ on $F_1$ is convex:  $$H_k^{1/2k}(u+v)\leq H_k^{1/2k}(u)+H_k^{1/2k}(v) \text{ for } u,v\in F_1. $$ Then the Finsler metric $(H_k\otimes H^k)^{1/2k}$ on $F_1\otimes F_2$ is convex.
\end{lemma}
Since $F_2$ is a line bundle, there is a canonical isomorphism between the bundles $S^k(F_1\otimes F_2)$ and $S^kF_1\otimes F_2^k$ which we use implicitly in the statement of Lemma \ref{convex}. Roughly speaking, Lemma \ref{convex} indicates that convexity is not affected by tensoring with a line bundle.

\begin{proof}
Fix $p\in X$. The fiber $F_2|_p$ is a one dimensional vector space, and we let $e$ be a basis.
For $x$ and $y\in F_1\otimes F_2 |_p$, we can write $x=\tilde{x}\otimes e$ and $y=\tilde{y}\otimes e$ where $\tilde{x},\tilde{y}\in F_1|_p$. By definition,
\begin{equation*}
   \begin{aligned}
    &(H_k\otimes H^k)^{\frac{1}{2k}}(x+y)=H_k\otimes H^k\big((x+y)^k,(x+y)^k   \big)^{\frac{1}{2k}}\\
    =&H_k\otimes H^k\big((\tilde{x}+\tilde{y})^k\otimes e^k ,(\tilde{x}+\tilde{y})^k\otimes e^k   \big)^{\frac{1}{2k}}\\
    =&H_k( (\tilde{x}+\tilde{y})^k,(\tilde{x}+\tilde{y})^k     )^{\frac{1}{2k}} H^k(e^k,e^k)^{\frac{1}{2k}}\\
    \leq & \big[H_k( \tilde{x}^k,\tilde{x}^k     )^{\frac{1}{2k}}+H_k( \tilde{y}^k,\tilde{y}^k     )^{\frac{1}{2k}}\big] H^k(e^k,e^k)^{\frac{1}{2k}}\\
    =&(H_k\otimes H^k)^{\frac{1}{2k}}(x)+(H_k\otimes H^k)^{\frac{1}{2k}}(y).
   \end{aligned} 
\end{equation*}
Therefore the Finsler metric $(H_k\otimes H^k)^{1/2k}$ is convex.\end{proof}

\subsection{Direct image bundles}\label{subsec direct}

We recall how to construct Hermitian metrics on direct image bundles and compute their curvature. Let $g$ be a Hermitian metric on $O_{P(E^*)}(1)$ with curvature $\Theta(g)$. Denote the restriction of the curvature to a fiber, $\Theta(g)|_{P(E^*_z)}$ by
$\omega_z$ for $z\in X$, and assume $\omega_z>0$ for all $z\in X$. With the canonical isomorphism $$\Phi_{k,z}:S^k E_z\to H^0(P(E^*_z),O_{P(E^*_z)}(k)), \text{ for $k\geq 0$ } $$ (see \cite[Page 278, Theorem 15.5]{demailly1997complex}), we define a Hermitian metric $H_k$ on $S^kE$ by 
\begin{align}\label{L2 metric}
H_k(u,v):=\int_{P(E^*_z)} g^k(\Phi_{k,z}(u),\Phi_{k,z}(v) ) \omega_z^{r-1}, \textup{ for $u \textup{ and } v\in S^kE_z$.}
\end{align}

Let us denote by $\Theta_k$ the curvature of $H_k$. Fixing $z\in X$ and $u\in S^k E_{z}$, in order to estimate the $(1,1)$-form $H_k(\Theta_k u,u)$, we first extend the vector $u$ to a local holomorphic section $\tilde{u}$ whose covariant derivative at $z$ with respect to $H_k$ equals zero. A straightforward computation shows $$\partial \bar{\partial} H_k( \tilde{u},\tilde{u}  )\big|_{z}=-H_k(\Theta_k u, u). $$ 
But $H_k( \tilde{u},\tilde{u}  )(\mathcal{z} ) $ for $\mathcal{z}$ near $z$ can also be written as the push-forward $$q_* \big( g^k( \Phi_{k,\mathcal{z}}(\tilde{u}),  \Phi_{k,\mathcal{z}}(\tilde{u})       )   \Theta(g)^{r-1} \big),   $$
where $q:P(E^*)\to X$ is the projection, so 
\begin{equation}\label{direct im}
   -H_k(\Theta_k u, u)=\partial \bar{\partial} H_k( \tilde{u},\tilde{u}  )|_z=   q_* \partial \bar{\partial}\big( g^k( \Phi_{k,\mathcal{z}}(\tilde{u}),  \Phi_{k,\mathcal{z}}(\tilde{u})      )   \Theta(g)^{r-1} \big)\big|_z. 
\end{equation}

Similarly, we can use a metric on $O_{P(E)}(1)\otimes p^*\det E$ to construct Hermitian metrics on $S^kE^* \otimes (\det E)^k$. The formula is similar to (\ref{L2 metric}), and we use bold symbols to highlight the change. Let $g$ be a metric on $O_{P(E)}(1)\otimes p^*\det E$ with curvature $\Theta(g)$. Denote the restriction of the curvature to a fiber $\Theta(g)|_{P(E_z)}$ by $\boldsymbol{\omega}_z$ for $z\in X$. Assume $\boldsymbol{\omega}_z>0$ for all $z\in X$. With the canonical isomorphism $$\boldsymbol{\Phi}_{k,z}:S^k E^*_z\otimes (\det E_z)^k\to H^0(P(E_z),O_{P(E_z)}(k)\otimes (p^*\det E_z)^k), \text{ for $k\geq 0$, } $$ we define a Hermitian metric $\boldsymbol{H}_k$ on $S^kE^*\otimes (\det E)^k$ by 
\begin{align}\label{L2 metric 3}
\boldsymbol{H}_k(u,v):=\int_{P(E_z)} g^k(\boldsymbol{\Phi}_{k,z}(u),\boldsymbol{\Phi}_{k,z}(v) ) \boldsymbol{\omega}_z^{r-1}, \textup{ for $u \textup{ and } v\in S^kE^*_z\otimes (\det E_z)^k$.}
\end{align} 
We also have a curvature formula similar to (\ref{direct im}).

\subsection{Berndtsson's positivity theorem}

Let $h$ be a metric on $O_{P(E^*)}(1)$ with curvature $\Theta(h)>0$. Denote $\Theta(h)|_{P(E^*_z)}$ by $\omega_z$ for $z\in X$. We are going to define a Hermitian metric on $\det E$ using the metric $h$. The relative canonical bundle $K_{P(E^*)/X}$ has a metric induced from $\{\omega_z^{r-1}\}_{z\in X}$. With $h^r$ on $O_{P(E^*)}(r)$ and the isomorphism $K_{P(E^*)/X}\otimes O_{P(E^*)}(r)\simeq q^* \det E$, there is an induced metric $\rho$ on $q^* \det E$. Using the canonical isomorphism $$\Psi_{z}:  \det E_z  \to H^0(P(E^*_z), q^*\det E_z), $$
 we define a Hermitian metric $H$ on $\det E$ by
\begin{align}\label{L2 metric 2}
    H(u,v):=\int_{P(E^*_z)} \rho (\Psi_{z}(u),\Psi_{z}(v) ) \omega_z^{r-1}, \textup{ for $u \textup{ and } v\in  \det E_z$.}
\end{align}
By Berndtsson's theorem \cite{Berndtsson09}, this metric $H$ is Griffiths positive, but it is the inequality that leads to this fact we will use. We follow the presentation in \cite[Section 4.1]{positivityandvanishingthmliu} (see also \cite[Section 2]{berndtssonconvexity}). Denote the curvature of $H$ by $\Theta$. Fix $z\in X$, $v\in \det E_z$, and $\eta\in T^{1,0}_z X$. For a local holomorphic frame of $E^*$ around $z$, we denote by $\{\zeta_i\}$ the fiber coordinates with respect to this frame, and by $\{z_\alpha\}$ the local coordinates on $X$. Around $P(E^*_z)$ in $P(E^*)$, we have homogeneous coordinates $[\zeta_1,...,\zeta_r]$ which induce local coordinates $(w_1,...,w_{r-1})$. For a local frame $e^*$ of $O_{P(E^*)}(1)$, we denote $h(e^*,e^*)$ by $e^{-\phi}$ and write the tangent vector $\eta=\sum \eta_\alpha \partial/\partial z_\alpha$. The inequality that leads to Berndtsson's theorem is 
\begin{equation}\label{positive detE}
    -H(\Theta v,v)(\eta,\bar{\eta})\leq \int_{P(E^*_z)} \rho (\Psi_{z}(v),\Psi_{z}(v) )   r \sum_{\alpha,\beta} \big(  \sum_{i,j}\phi_{\alpha \bar{j}}\phi^{i\bar{j}}\phi_{i\bar{\beta}}-\phi_{\alpha\bar{\beta}}   \big)\eta_\alpha \bar{\eta}_\beta  \omega_z^{r-1},  
\end{equation}
where  $\phi_{i\bar{j}}:=\partial^2  \phi/\partial w_i \partial \bar{w}_j$, $\phi_{\alpha\bar{j}}:=\partial^2  \phi/\partial z_\alpha \partial \bar{w}_j$, $\phi_{\alpha\bar{\beta}}:=\partial^2  \phi/\partial z_\alpha \partial \bar{z}_\beta$, and $(\phi^{i\bar{j}})$ is the inverse matrix of $(\phi_{i\bar{j}})$. Since $\det E$ is a line bundle, the curvature $\Theta$ is a $(1,1)$-form, and so $H(\Theta v,v)(\eta,\bar{\eta})=H( v,v)\Theta(\eta,\bar{\eta})$. If we further assume $H(v,v)=1$, then the left hand side of (\ref{positive detE}) becomes $-\Theta(\eta,\bar{\eta})$.


\section{Proof of Theorem \ref{thm 1}}\label{sec pf of thm1}

We use the metric $h$ to construct a Hermitian metric $H$ on $\det E$ as in (\ref{L2 metric 2}), and the metric $g$ to construct Hermitian metrics $H_k$ on $S^kE$ as in (\ref{L2 metric}). The number $k$ is yet to be determined. 

We start with the metric $g$. Given a point $(z_0,[\zeta_0])\in P(E^*)$, we have the normal coordinate system from Subsection \ref{system}. In this coordinate system, let us introduce the following $n$-by-$n$ matrix-valued function
\begin{align*}
    B_k=\big((B_k)_{\alpha\beta}\big):=\big( k\phi_{\alpha\bar{\beta}}- (\log \det( \phi_{i\bar{j}}))_{\alpha\bar{\beta}}      \big),
\end{align*}
where $g(e^*,e^*)=e^{-\phi}$. By continuity, there is a neighborhood $U$ of $(z_0,[\zeta_0])$ in $P(E^*)$ such that in $U$
\begin{align}
    (\phi_{\alpha\bar{\beta}})|_{(z_0,[\zeta_0])}+\frac{r-M}{4}\Id_{n\times n}\geq (\phi_{\alpha\bar{\beta}}).\label{456}
\end{align}
For this $U$, there is a positive integer $k_0$ such that for $k\geq k_0$ and in $U$
\begin{equation}\label{789}
     (\phi_{\alpha\bar{\beta}})+\frac{r-M}{4}\Id_{n\times n}\geq \frac{B_k}{k}.
\end{equation}
Let us summarize what we have done so far in 
\begin{lemma}\label{lem sum}
Given a point $(z_0,[\zeta_0])\in P(E^*)$, there exist a coordinate neighborhood $U$ of $(z_0,[\zeta_0])$ in $P(E^*)$ and a positive integer $k_0$ such that in $U$ and for $k\geq k_0$
\begin{align}
(\phi_{\alpha\bar{\beta}})|_{(z_0,[\zeta_0])}+\frac{r-M}{2}\Id_{n\times n}\geq \frac{B_k}{k}.\label{ineq 2 in lem}
\end{align}
\end{lemma}

By Lemma \ref{lem sum}, since $P(E^*_{z_0})$ is compact, we can find finitely many points $\{(z_0,[\zeta_l])\}_l$ on $P(E^*_{z_0})$ each of which corresponds to a coordinate neighborhood $U_l$ in $P(E^*)$ and a positive integer $k_l$ such that the corresponding (\ref{ineq 2 in lem}) holds, and $P(E^*_{z_0})\subset\bigcup_l U_l$. Denote $\max_l k_l$ by $k_{\max}$. The point $z_0$ has a neighborhood $W$ in $X$ such that for $z\in W$, the fiber $P(E^*_z)$ can be partitioned as $\bigcup_m V_m$ with each $V_m$ in $U_l$ for some $l$. By shrinking $W$, we can assume that for each $U_l$ the corresponding $\Omega(\partial/\partial z_\alpha,\partial/\partial \bar{z}_\beta ):=\Omega_{\alpha\bar{\beta}}$ satisfies 
\begin{equation}\label{epsilon}
-\varepsilon \delta_{\alpha\beta}    <\Omega_{\alpha\bar{\beta}}(z)-\delta_{\alpha\beta}<\varepsilon \delta_{\alpha\beta} \text{, for $z\in W$}
\end{equation}
where $\varepsilon:=(r-M)/5(r+M)$.

Recall the Hermitian metrics $H_k$ on $S^kE$ in (\ref{L2 metric}) constructed using the metric $g$. Denote by $\Theta_k$ the curvature of $H_k$. We claim
\begin{lemma}\label{lem ineq 1}
For $k\geq k_{\max}$, $z\in W$, $0\neq \eta \in T^{1,0}_z X$, and $ u\in S^kE_z$ with $H_k(u,u)=1$, we have 
\begin{equation}
\begin{aligned}
H_k(\Theta_k u, u)(\eta,\bar{\eta})\leq  \big(M+\frac{r-M}{2} \big)k\frac{\Omega (\eta,\bar{\eta})}{(1-\varepsilon)}.
\end{aligned} 
\end{equation}
\end{lemma}

\begin{proof}
As in Subsection \ref{subsec direct}, we extend the vector $u\in S^kE_z$ to a local holomorphic section $\tilde{u}$ whose covariant derivative at $z$ equals zero, and we have
\begin{align*}
-H_k(\Theta_k u, u)=&\partial \bar{\partial} H_k( \tilde{u},\tilde{u}  )\big|_{z}=\int_{P(E^*_z)} \partial \bar{\partial}\big( g^k( \Phi_{k,\mathcal{z}}(\tilde{u}),  \Phi_{k,\mathcal{z}}(\tilde{u})      )   \Theta(g)^{r-1} \big)\\
=&\sum_m\int_{V_m} \partial \bar{\partial}\big( g^k( \Phi_{k,\mathcal{z}}(\tilde{u}),  \Phi_{k,\mathcal{z}}(\tilde{u})      )   \Theta(g)^{r-1} \big);
\end{align*}
in the last equality, we partition the fiber $P(E^*_z)$ as $\bigcup_m V_m$ with each $V_m$ in $U_l$ for some $l$. In a fixed $V_m\subset U_l$, using the coordinate system of $U_l$, we can write $\Phi_{k,\mathcal{z}}(\tilde{u})$ as $f(e^*)^k$ with $f$ a scalar-valued holomorphic function and $e^*$ a local frame for $O_{P(E^*)}(1)$. So, $g^k( \Phi_{k,\mathcal{z}}(\tilde{u}),  \Phi_{k,\mathcal{z}}(\tilde{u})      )=|f|^2 e^{-k\phi}$. Meanwhile, recall the curvature $\Theta(g)=\partial
\bar{\partial }\phi$. By Stokes' Theorem and a count on degrees, we have 
\begin{align*}
    &\sum_m\int_{V_m} \partial \bar{\partial}\big( g^k( \Phi_{k,\mathcal{z}}(\tilde{u}),  \Phi_{k,\mathcal{z}}(\tilde{u})      )   \Theta(g)^{r-1} \big)\\
    =&\sum_m\int_{V_m}  \sum_{\alpha, \beta}\frac{\partial^2 |f|^2e^{-k\phi}\det(\phi_{i\bar{j}})}{\partial z_\alpha \partial\bar{z}_\beta}dz_\alpha \wedge d\bar{z}_\beta \bigwedge_j dw_j\wedge d\bar{w}_j.
\end{align*}
So, if the tangent vector $\eta=\sum_\alpha \eta_\alpha \partial/\partial z_\alpha$ in the coordinate neighborhood $U_l$, then    
\begin{equation}\label{curvature partition}
    -H_k(\Theta_k u, u)(\eta,\bar{\eta})=\sum_m\int_{V_m}  \sum_{\alpha, \beta}\frac{\partial^2 |f|^2e^{-k\phi}\det(\phi_{i\bar{j}})}{\partial z_\alpha \partial\bar{z}_\beta} \eta_\alpha  \bar{\eta}_\beta \bigwedge_j dw_j\wedge d\bar{w}_j.
\end{equation}
Note that the integrands in (\ref{curvature partition}) are written in the local coordinates of corresponding $U_l$.
A direct computation shows
\begin{align*}
    \sum_{\alpha, \beta}\frac{\partial^2 |f|^2e^{-k\phi}\det(\phi_{i\bar{j}})}{\partial z_\alpha \partial\bar{z}_\beta} \eta_\alpha  \bar{\eta}_\beta
    =&e^{-k\phi}\det (\phi_{i\bar{j}})  \big|  \sum_\alpha \frac{\partial f}{\partial  z_\alpha}\eta_\alpha-f\sum_\alpha \big(k\phi_{\alpha}-(\log\det\phi_{i\bar{j}})_\alpha    \big)\eta_\alpha      \big|^2\\
    &- |f|^2 e^{-k\phi}\det (\phi_{i\bar{j}}) \sum_{\alpha,\beta} \big(  k\phi_{\alpha\bar{\beta}}-(\log\det\phi_{i\bar{j}})_{\alpha\bar{\beta}}      \big)\eta_\alpha \bar{\eta}_\beta\\
    \geq& -|f|^2 e^{-k\phi}\det (\phi_{i\bar{j}}) \sum_{\alpha,\beta}   (B_k)_{\alpha \beta}      \eta_\alpha \bar{\eta}_\beta.
\end{align*}
By (\ref{ineq 2 in lem}),     
\begin{equation}\label{after 3.8}
    \frac{1}{k}\sum_{\alpha,\beta}   (B_k)_{\alpha \beta}      \eta_\alpha \bar{\eta}_\beta\leq   \sum_{\alpha, \beta}    \phi_{\alpha\bar{\beta}}|_{(z_0,[\zeta_l])}\eta_\alpha\bar{\eta}_\beta          + \frac{r-M}{2} \sum_\alpha |\eta_\alpha|^2.
\end{equation}
Using the coordinate system of $U_l$, the tangent vector $\eta=\sum_\alpha \eta_\alpha \partial /\partial z_\alpha$ at $z$ induces a tangent vector $\eta_l=\sum_\alpha \eta_\alpha \partial /\partial z_\alpha|_{z_0}$ at $z_0$. Denote the lifts of $\eta_l$ to $T^{1,0}_{(z_0,[\zeta_l])}P(E^*)$ by $\tilde{\eta}_l$. According to (\ref{inf=kob}), (\ref{thm 1 assumption}), and (\ref{min}), we see 
\begin{equation}
M\sum_\alpha |\eta_\alpha|^2    \geq  -\theta(g)(\tilde{\eta}_l, \bar{\tilde{\eta}}_l) =\inf_{q_*(\tilde{\eta}_l)=\eta_l} \Theta(g)(\tilde{\eta}_l,\bar{\tilde{\eta}}_l)=\sum_{\alpha,\beta} \phi_{\alpha\bar{\beta}}|_{(z_0,[\zeta_l])}\eta_\alpha \bar{\eta}_\beta.
\end{equation}
Therefore, (\ref{after 3.8}) becomes
\begin{equation}
    \frac{1}{k}\sum_{\alpha,\beta}   (B_k)_{\alpha \beta}      \eta_\alpha \bar{\eta}_\beta\leq     \big(M        + \frac{r-M}{2} \big)\sum_\alpha |\eta_\alpha|^2\leq      \big(M+\frac{r-M}{2} \big)\frac{\Omega (\eta,\bar{\eta})}{(1-\varepsilon)},
\end{equation}
where we use (\ref{epsilon}) in the second inequality. So, (\ref{curvature partition}) becomes
\begin{equation}\label{first term}
\begin{aligned}
-H_k(\Theta_k u, u)(\eta,\bar{\eta})&\geq   \sum_m\int_{V_m}   -|f|^2 e^{-k\phi}\det (\phi_{i\bar{j}})\bigwedge_j dw_j\wedge d\bar{w}_j \big(M+\frac{r-M}{2} \big)k\frac{\Omega (\eta,\bar{\eta})}{(1-\varepsilon)}\\
&=-\big(M+\frac{r-M}{2} \big)k\frac{\Omega (\eta,\bar{\eta})}{(1-\varepsilon)}
\end{aligned} 
\end{equation}
since $H_k(u,u)=1$.
\end{proof}

We turn now to the metric $h$. The argument about $h$ is similar to that about $g$, and it will be used in Theorems \ref{thm 2}, \ref{thm 3}, and \ref{thm 4}. Given a point $(z_0,[\zeta_0])\in P(E^*)$, we have the normal coordinate system from Subsection \ref{system} with respect to the metric $h$. In this coordinate system, let us introduce the $n$-by-$n$ matrix-valued function
\begin{align*}
    A=(A_{\alpha\beta}):=\big(  \phi_{\alpha\bar{\beta}}-\sum_{i,j}\phi_{\alpha \bar{j}}\phi^{i\bar{j}}\phi_{i\bar{\beta}}           \big),
\end{align*}
where $h(e^*,e^*)=e^{-\phi}$ and $(\phi^{i\bar{j}})$ is the inverse matrix of $(\phi_{i\bar{j}})$. By continuity, there is a neighborhood $U$ of $(z_0,[\zeta_0])$ in $P(E^*)$ such that in $U$
\begin{align}
    rA+\frac{r-M}{4}\Id_{n\times n}  \geq rA|_{(z_0,[\zeta_0])}.\label{123'} 
\end{align}
In summary, 
\begin{lemma}\label{lem sum'}
Given a point $(z_0,[\zeta_0])\in P(E^*)$, there exists a coordinate neighborhood $U$ of $(z_0,[\zeta_0])$ in $P(E^*)$ such that in $U$
\begin{align}
rA+\frac{r-M}{4}\Id_{n\times n}  \geq r(\phi_{\alpha\bar{\beta}})|_{(z_0,[\zeta_0])}.\label{ineq 1 in lem'}
\end{align}
\end{lemma}

By Lemma \ref{lem sum'}, since $P(E^*_{z_0})$ is compact, we can find finitely many points $\{(z_0,[\zeta_l])\}_l$ on $P(E^*_{z_0})$ each of which corresponds to a coordinate neighborhood $U_l$ in $P(E^*)$ such that the corresponding (\ref{ineq 1 in lem'}) holds, and $P(E^*_{z_0})\subset\bigcup_l U_l$. The point $z_0$ has a neighborhood $W'$ in $X$ such that for $z\in W'$, the fiber $P(E^*_z)$ can be partitioned as $\bigcup_m V_m$ with each $V_m$ in $U_l$ for some $l$. By shrinking $W'$, we can assume that for each $U_l$ the corresponding $\Omega(\partial/\partial z_\alpha,\partial/\partial \bar{z}_\beta ):=\Omega_{\alpha\bar{\beta}}$ satisfies 
\begin{equation}\label{epsilon'}
-\varepsilon \delta_{\alpha\beta}    <\Omega_{\alpha\bar{\beta}}(z)-\delta_{\alpha\beta}<\varepsilon \delta_{\alpha\beta} \text{, for $z\in W'$}
\end{equation}
where $\varepsilon:=(r-M)/5(r+M)$.

Recall the Hermitian metric $H$ on $\det E$ in (\ref{L2 metric 2}) constructed using the metric $h$. Denote by $\Theta$ the curvature of $H$. We claim
\begin{lemma}\label{lem ineq 2}
For $z\in W'$ and $\eta\in T^{1,0}_zX$, we have 
\begin{equation}
    -\Theta(\eta,\bar{\eta})\leq -\big( r-\frac{r-M}{4} \big) \frac{\Omega(\eta,\bar{\eta})}{(1+\varepsilon)}.
\end{equation}
\end{lemma}

\begin{proof}
Using (\ref{positive detE}) and assuming $H(v,v)=1$, we get 
\begin{align}\label{positive detE partition}
    -\Theta(\eta,\bar{\eta})\leq \sum_m\int_{V_m} \rho (\Psi_{z}(v),\Psi_{z}(v) )   r \sum_{\alpha,\beta} \big(  \sum_{i,j}\phi_{\alpha \bar{j}}\phi^{i\bar{j}}\phi_{i\bar{\beta}}-\phi_{\alpha\bar{\beta}}   \big)\eta_\alpha \bar{\eta}_\beta  \omega_z^{r-1},  
\end{align}
where we again partition $P(E^*_z)$ as $\bigcup_m V_m$ with each $V_m$ in $U_l$ for some $l$. Note that the integrands in (\ref{positive detE partition}) are written in the local coordinates of corresponding $U_l$. In a fixed $V_m\subset U_l$, we have $\eta=\sum_\alpha \eta_\alpha \partial/\partial z_\alpha$, and by (\ref{ineq 1 in lem'}) we see     
\begin{equation}\label{ineq i dont know}
    r\sum_{\alpha,\beta} A_{\alpha\beta}\eta_\alpha \bar{\eta}_\beta+\frac{r-M}{4}\sum_{\alpha}|\eta_\alpha|^2\geq r\sum_{\alpha, \beta}    \phi_{\alpha\bar{\beta}}|_{(z_0,[\zeta_l])}\eta_\alpha\bar{\eta}_\beta. 
\end{equation}
In $U_l$, the tangent vector $\eta=\sum_\alpha \eta_\alpha \partial /\partial z_\alpha$ at $z$ induces a tangent vector $\eta_l=\sum_\alpha \eta_\alpha \partial /\partial z_\alpha|_{z_0}$ at $z_0$. Denote the lifts of $\eta_l$ to $T^{1,0}_{(z_0,[\zeta_l])}P(E^*)$ by $\tilde{\eta}_l$. By (\ref{inf=kob}), (\ref{thm 1 assumption 2}), and (\ref{min}), we see 
\begin{equation}
\sum_\alpha |\eta_\alpha|^2    \leq -\theta(h)(\tilde{\eta}_l,\bar{\tilde{\eta}}_l) =\inf_{q_*(\tilde{\eta}_l)=\eta_l} \Theta(h)(\tilde{\eta}_l,\bar{\tilde{\eta}}_l)=\sum_{\alpha,\beta} \phi_{\alpha\bar{\beta}}|_{(z_0,[\zeta_l])}\eta_\alpha \bar{\eta}_\beta.
\end{equation}
Therefore, (\ref{ineq i dont know}) becomes
\begin{equation*}
     r\sum_{\alpha,\beta} A_{\alpha\beta}\eta_\alpha \bar{\eta}_\beta\geq \big( r-\frac{r-M}{4} \big)\sum_\alpha |\eta_\alpha|^2 \geq \big( r-\frac{r-M}{4} \big) \frac{\Omega(\eta,\bar{\eta})}{1+\varepsilon},
\end{equation*}
where we use (\ref{epsilon'}) in the second inequality. So, (\ref{positive detE partition}) becomes
\begin{equation*}\label{second term}
    -\Theta(\eta,\bar{\eta})\leq -\sum_m\int_{V_m} \rho (\Psi_{z}(v),\Psi_{z}(v) )   \omega_z^{r-1}   \big( r-\frac{r-M}{4} \big) \frac{\Omega(\eta,\bar{\eta})}{1+\varepsilon}= -\big( r-\frac{r-M}{4} \big) \frac{\Omega(\eta,\bar{\eta})}{(1+\varepsilon)}
\end{equation*}
because $H(v,v)=1$. 
\end{proof}

Now we put together the $L^2$-metrics $H_k$ on $S^kE$ in (\ref{L2 metric}), and $H$ on $\det E$ in (\ref{L2 metric 2}). Since $(\det E^*)^k$ is a line bundle, we can identify $\End (S^kE\otimes (\det E^*)^k)$ with $\End (S^kE)$, and the curvature of the metric $H_k\otimes (H^*)^k$ on $S^kE\otimes (\det E^*)^k$ can be written as $$\Theta_k-k\Theta\otimes \Id_{S^kE},$$ where $\Theta_k$ and $\Theta$ are the curvature of $H_k$ and $H$ respectively. We claim that for $k\geq k_{\max}$ and in $W\bigcap W'$ a neighborhood of $z_0$, the metric $H_k\otimes (H^*)^k$ is Griffiths negative. Indeed, as a result of Lemmas \ref{lem ineq 1}  and \ref{lem ineq 2}, for $k\geq k_{\max}$, $z\in W\bigcap W'$, $0\neq \eta \in T^{1,0}_z X$, and $ u\in S^kE_z$ with $H_k(u,u)=1$, we see
\begin{equation*}
    H_k(\Theta_k u, u)(\eta,\bar{\eta})-k\Theta(\eta,\bar{\eta})\leq  k\big(M+\frac{r-M}{2} \big)\frac{\Omega (\eta,\bar{\eta})}{(1-\varepsilon)}     -k\big( r-\frac{r-M}{4} \big) \frac{\Omega(\eta,\bar{\eta})}{(1+\varepsilon)};
\end{equation*}
the term on the right is negative after some computation using $\varepsilon=(r-M)/5(r+M)$. So, we have proved the claim that for $k\geq k_{\max}$ and in $W\bigcap W'\subset X$, the metric $H_k\otimes (H^*)^k$ is Griffiths negative. Since $X$ is compact, $H_k\otimes (H^*)^k$ is Griffiths negative on the entire $X$ for $k$ large enough.

Now we fix $k$ such that the Hermitian metric $H_k\otimes (H^*)^k$ on the bundle $S^kE \otimes (\det E^*)^k$ is Griffiths negative on $X$. The Hermitian metric $H_k$ by construction is an $L^2$-integral, so its $k$-th root is a convex Finsler metric on $E$ (see \cite[Proof of Theorem 1]{wu_2022} for details). By Lemma \ref{convex}, the $k$-th root of $H_k\otimes (H^*)^k$ is a convex Finsler metric on $E \otimes \det E^*$ which we denote by $F$. Moreover, this Finsler metric $F$ is strongly plurisubharmonic on $E \otimes \det E^* \setminus \{\text{zero section}\}$ due to Griffiths negativity of $H_k\otimes (H^*)^k$. By adding a small Hermitian metric, we can assume $F$ is strongly convex and strongly plurisubharmonic. As in \cite[Proof of Theorem 1]{wu_2022}, the dual Finsler metric of $F$ is a convex, strongly pseudoconvex, and Kobayashi positive Finsler metric on $E^* \otimes \det E$, hence the proof of Theorem \ref{thm 1} is complete.

\section{Proof of Theorem \ref{thm 2}}\label{sec pf o thm2}
The proof is similar to what we do in Section \ref{sec pf of thm1} except that we are dealing with not only $P(E^*)$ but $P(E)$ here. The metric $h$ is used to define a Hermitian metric $H$ on $\det E$ as in (\ref{L2 metric 2}). The metric $g$ is used to define Hermitian metrics $\boldsymbol{H}_k$ on $S^kE^*\otimes (\det E)^k$ as in (\ref{L2 metric 3}).

Fix $z_0$ in $X$. For the metric $h$ on $O_{P(E^*)}(1)$, we follow the path that leads to Lemma \ref{lem ineq 2} in Section \ref{sec pf of thm1} to deduce a neighborhood $W'$ of $z_0$ in $X$ such that for $z\in W'$ and $\eta\in T^{1,0}_zX$, the curvature $\Theta$ of $H$ satisfies 
\begin{equation}\label{one}
    -\Theta(\eta,\bar{\eta})\leq -\big( r-\frac{r-M}{4} \big) \frac{\Omega(\eta,\bar{\eta})}{(1+\varepsilon)}
\end{equation}
with $\varepsilon=(r-M)/5(r+M)$.

For the metric $g$ on $O_{P(E)}(1)\otimes p^*\det E$, we replace $O_{P(E^*)}(1)\to P(E^*)$ in Section \ref{sec pf of thm1} with $O_{P(E\otimes \det E^*)}(1)\to P(E\otimes \det E^*)$ and use the canonical isomorphism between $O_{P(E\otimes \det E^*)}(1)\to P(E\otimes \det E^*)$ and $O_{P(E)}(1)\otimes p^*\det E \to P(E)$. Then following the argument leading to Lemma \ref{lem ineq 1}, we obtain a positive integer $k_{\max}$ and a neighborhood $W$ of $z_0$ in $X$ such that for $k\geq k_{\max}$, $z\in W$, $\eta\in T^{1,0}_zX$, and $u\in S^kE_z^*\otimes (\det E_z)^k$ with $\boldsymbol{H}_k(u,u)=1$, the curvature $\boldsymbol{\Theta}_k$ of $\boldsymbol{H}_k$ satisfies
\begin{equation}\label{two}
\boldsymbol{H}_k(\boldsymbol{\Theta}_k u, u)(\eta,\bar{\eta})\leq \big(M+\frac{r-M}{2} \big)k\frac{\Omega (\eta,\bar{\eta})}{(1-\varepsilon)}.
\end{equation}

On the bundle $[S^kE^*\otimes (\det E)^k]\otimes (\det E^*)^k$, there is a Hermitian metric $\boldsymbol{H}_k\otimes (H^*)^k$ with curvature $\boldsymbol{\Theta}_k-k\Theta\otimes \Id_{S^kE^*\otimes (\det E)^k}$. As a result of (\ref{one}) and (\ref{two}), we deduce that, for $k\geq k_{\max }$, $z\in W\bigcap W'$, $\eta\in T^{1,0}_zX$, and $u\in S^kE_z^*\otimes (\det E_z)^k$ with $\boldsymbol{H}_k(u,u)=1$,
\begin{equation*}
      \boldsymbol{H}_k(\boldsymbol{\Theta}_k u, u)(\eta,\bar{\eta})-k\Theta(\eta,\bar{\eta}) \leq  k\big(M+\frac{r-M}{2} \big)\frac{\Omega (\eta,\bar{\eta})}{(1-\varepsilon)}     -k\big( r-\frac{r-M}{4} \big) \frac{\Omega(\eta,\bar{\eta})}{(1+\varepsilon)};
\end{equation*}
again, the term on the right is negative using $\varepsilon=(r-M)/5(r+M)$. So we have proved that for $k\geq k_{\max }$ and in $W\bigcap W'$, the metric $\boldsymbol{H}_k\otimes (H^*)^k$ is Griffiths negative. Since $X$ is compact, $\boldsymbol{H}_k\otimes (H^*)^k$ is Griffiths negative on $X$ for $k$ large.

Now we fix $k$ such that $\boldsymbol{H}_k\otimes (H^*)^k$ on the bundle $[S^kE^*\otimes (\det E)^k]\otimes (\det E^*)^k\simeq S^kE^* $ is Griffiths negative. Using the same argument as in the last paragraph in Section \ref{sec pf of thm1}, we obtain a convex, strongly pseudoconvex, Kobayashi positive Finsler metric on $E$.

\section{Proof of Theorem \ref{thm 3}}\label{sec pf of thm3}
We use the metric $h$ to construct a Hermitian metric $H$ on $\det E$ as in (\ref{L2 metric 2}), and the metric $g$ to construct a Hermitian metric $H_1$ on $S^1E=E$ as in (\ref{L2 metric}).

We start with the metric $g$. For $(z_0,[\zeta_0])$ in $P(E^*)$, there is a special coordinate system given in Subsection \ref{sub 2.5}. In this coordinate system, we define the following $n$-by-$n$ matrix-valued function
\begin{align*}
        B=(B_{\alpha\beta}):=\big( \phi_{\alpha\bar{\beta}}- (\log \det( \phi_{i\bar{j}}))_{\alpha\bar{\beta}}      \big),
\end{align*}
where $g(e^*,e^*)=e^{-\phi}$.
By continuity, there is a neighborhood $U$ of $(z_0,[\zeta_0])$ in $P(E^*)$ such that in $U$
\begin{align*}
       B|_{(z_0,[\zeta_0])}+\frac{r-M}{4}\Id_{n\times n}\geq B.
\end{align*}
In summary,
\begin{lemma}\label{lem sum''}
Given a point $(z_0,[\zeta_0])\in P(E^*)$, there exists a coordinate neighborhood $U$ of $(z_0,[\zeta_0])$ in $P(E^*)$ such that in $U$ 
\begin{align}
B|_{(z_0,[\zeta_0])}+\frac{r-M}{4}\Id_{n\times n}\geq B.\label{ineq 2 in lem'}
\end{align}
\end{lemma}
By Lemma \ref{lem sum''}, since $P(E^*_{z_0})$ is compact, we can find finitely many points $\{(z_0,[\zeta_l])\}_l$ on $P(E^*_{z_0})$ each of which corresponds to a coordinate neighborhood $U_l$ in $P(E^*)$ such that the corresponding (\ref{ineq 2 in lem'}) holds, and $P(E^*_{z_0})\subset \bigcup_l U_l$. The fiber $P(E^*_{z_0})$ can be partitioned as $\bigcup_m V_m$ with each $V_m$ in $U_l$ for some $l$.
 
Recall the Hermitian metric $H_1$ on $E$ in (\ref{L2 metric}) constructed using the metric $g$. Denote by $\Theta_1$ the curvature of $H_1$. We claim 
\begin{lemma}\label{what ever}
For $0\neq \eta \in T^{1,0}_{z_0} X$ and $ u\in E_{z_0}$ with $H_1(u,u)=1$, we have 
\begin{equation}
H_1(\Theta_1 u, u)(\eta,\bar{\eta})\leq   \big(M+\frac{r-M}{4} \big)\Omega (\eta,\bar{\eta}).
\end{equation}
\end{lemma}
\begin{proof}
As in Subsection \ref{subsec direct}, we extend the vector $u\in E_{z_0}$ to a local holomorphic section $\tilde{u}$ whose covariant derivative at $z_0$ equals zero, and we have
\begin{align*}
-H_1(\Theta_1 u, u)=\partial \bar{\partial} H_1( \tilde{u},\tilde{u}  )\big|_{z_0}=&\int_{P(E^*_{z_0})} \partial \bar{\partial}\big( g( \Phi_{1,\mathcal{z}}(\tilde{u}),  \Phi_{1,\mathcal{z}}(\tilde{u})      )   \Theta(g)^{r-1} \big)\\
=&\sum_m\int_{V_m} \partial \bar{\partial}\big( g( \Phi_{1,\mathcal{z}}(\tilde{u}),  \Phi_{1,\mathcal{z}}(\tilde{u})      )   \Theta(g)^{r-1} \big).
\end{align*}
In a fixed $V_m \subset U_l$, we can write $\Phi_{1,\mathcal{z}}(\tilde{u})$ as $fe^*$ with $f$ a scalar-valued holomorphic function and $e^*$ a local frame for $O_{P(E^*)}(1)$. So, $g( \Phi_{1,\mathcal{z}}(\tilde{u}),  \Phi_{1,\mathcal{z}}(\tilde{u})      )=|f|^2 e^{-\phi}$. Meanwhile, recall the curvature $\Theta(g)=\partial
\bar{\partial }\phi$. By Stokes' Theorem and a count on degrees, we have 
\begin{align*}
    &\sum_m\int_{V_m} \partial \bar{\partial}\big( g( \Phi_{1,\mathcal{z}}(\tilde{u}),  \Phi_{1,\mathcal{z}}(\tilde{u})      )   \Theta(g)^{r-1} \big)\\
    =&\sum_m\int_{V_m}  \sum_{\alpha, \beta}\frac{\partial^2 |f|^2e^{-\phi}\det(\phi_{i\bar{j}})}{\partial z_\alpha \partial\bar{z}_\beta}dz_\alpha \wedge d\bar{z}_\beta \bigwedge_j dw_j\wedge d\bar{w}_j.
\end{align*}
So,     
\begin{equation}\label{curvature partition'}
    -H_1(\Theta_1 u, u)(\eta,\bar{\eta})=\sum_m\int_{V_m}  \sum_{\alpha, \beta}\frac{\partial^2 |f|^2e^{-\phi}\det(\phi_{i\bar{j}})}{\partial z_\alpha \partial\bar{z}_\beta} \eta_\alpha  \bar{\eta}_\beta \bigwedge_j dw_j\wedge d\bar{w}_j,
\end{equation}
for $T^{1,0}_{z_0}X \ni \eta=\sum_\alpha \eta_\alpha \partial /\partial z_\alpha$. A direct computation shows
\begin{align*}
    \sum_{\alpha, \beta}\frac{\partial^2 |f|^2e^{-\phi}\det(\phi_{i\bar{j}})}{\partial z_\alpha \partial\bar{z}_\beta} \eta_\alpha  \bar{\eta}_\beta
    =&e^{-\phi}\det (\phi_{i\bar{j}})  \big|  \sum_\alpha \frac{\partial f}{\partial  z_\alpha}\eta_\alpha-f\sum_\alpha \big(\phi_{\alpha}-(\log\det\phi_{i\bar{j}})_\alpha    \big)\eta_\alpha      \big|^2\\
    &- |f|^2 e^{-\phi}\det (\phi_{i\bar{j}}) \sum_{\alpha,\beta} \big(  \phi_{\alpha\bar{\beta}}-(\log\det\phi_{i\bar{j}})_{\alpha\bar{\beta}}      \big)\eta_\alpha \bar{\eta}_\beta\\
    \geq& -|f|^2 e^{-\phi}\det (\phi_{i\bar{j}}) \sum_{\alpha,\beta}   B_{\alpha \beta}      \eta_\alpha \bar{\eta}_\beta.
\end{align*}
By (\ref{sup}), (\ref{thm 3 assumption}), (\ref{min 2}), and (\ref{min'}), we see \begin{align*}
M\sum_\alpha |\eta_\alpha|^2    \geq& -(r+1)\theta(g)(\tilde{\eta},\bar{\tilde{\eta}})+q^*\Theta(\det G)(\tilde{\eta},\bar{\tilde{\eta}})\\
=&\inf_{q_*(\tilde{\eta})=\eta} \Theta(g)(\tilde{\eta},\bar{\tilde{\eta}})-
\sup_{q_*(\tilde{\eta})=\eta} \gamma_g (\tilde{\eta},\bar{\tilde{\eta}})\\
=&\sum_{\alpha,\beta} \phi_{\alpha\bar{\beta}}|_{(z_0,[\zeta_l])}\eta_\alpha \bar{\eta}_\beta-\sum_{\alpha,\beta} (\log \det (\phi_{i\bar{j}}))_{\alpha\bar{\beta}}|_{(z_0,[\zeta_l])} \eta_\alpha \bar{\eta}_\beta\\
=&\sum_{\alpha,\beta} B_{\alpha \beta}|_{(z_0,[\zeta_l])}\eta_\alpha \bar{\eta}_\beta.
\end{align*}
Therefore, (\ref{ineq 2 in lem'}) becomes
\begin{equation}
    \sum_{\alpha,\beta}   B_{\alpha \beta}      \eta_\alpha \bar{\eta}_\beta\leq     \big(M        + \frac{r-M}{4} \big)\sum_\alpha |\eta_\alpha|^2=      \big(M+\frac{r-M}{4} \big)\Omega (\eta,\bar{\eta}).
\end{equation}
So (\ref{curvature partition'}) becomes
\begin{equation}\label{first term'}
\begin{aligned}
-H_1(\Theta_1 u, u)(\eta,\bar{\eta})&\geq   \sum_m\int_{V_m}   -|f|^2 e^{-\phi}\det (\phi_{i\bar{j}})\bigwedge_j dw_j\wedge d\bar{w}_j \big(M+\frac{r-M}{4} \big)\Omega (\eta,\bar{\eta})\\
&=-\big(M+\frac{r-M}{4} \big)\Omega (\eta,\bar{\eta})
\end{aligned} 
\end{equation}
since $H_1(u,u)=1$. 
\end{proof}

For the metric $h$ on $O_{P(E^*)}(1)$, as in Lemma \ref{lem ineq 2} from Section \ref{sec pf of thm1} with slight modification, we deduce that for $\eta\in T^{1,0}_{z_0}X$, the curvature $\Theta$ of $H$ satisfies 
\begin{equation}\label{whatever}
    -\Theta(\eta,\bar{\eta})\leq -\big( r-\frac{r-M}{4} \big) \Omega(\eta,\bar{\eta}).
\end{equation}

Finally, we consider the metric $H_1\otimes H^*$ on $E\otimes \det E^*$. Since $\det E^*$ is a line bundle, we can identify $\End (E\otimes \det E^*)$ with $\End E$, and the curvature of the metric $H_1\otimes H^*$ can be written as $$\Theta_1-\Theta\otimes \Id_{E},$$ where $\Theta_1$ and $\Theta$ are the curvature of $H_1$ and $H$ respectively. 
As a result of Lemma \ref{what ever} and (\ref{whatever}), we see for $0\neq \eta \in T^{1,0}_{z_0} X$ and $ u\in E_{z_0}$ with $H_1(u,u)=1$,
\begin{equation*}
    H_1(\Theta_1 u, u)(\eta,\bar{\eta})-\Theta(\eta,\bar{\eta})\leq \big(M+\frac{r-M}{4} \big)\Omega (\eta,\bar{\eta})     -\big( r-\frac{r-M}{4} \big) \Omega(\eta,\bar{\eta});
\end{equation*}
the term on the right is negative. Hence we have proved that at $z_0$ the metric $H_1\otimes H^*$ is Griffiths negative. The point $z_0$ is arbitrary, so $H_1\otimes H^*$ is Griffiths negative on $X$. As a result, the dual bundle $E^*\otimes \det E$ is Griffiths positive.

\section{Proof of Theorem \ref{thm 4}}\label{sec pf of thm 4}
The metric $h$ is used to define a Hermitian metric $H$ on $\det E$ as in (\ref{L2 metric 2}). The metric $g$ is used to define Hermitian metric $\boldsymbol{H}_1$ on $E^*\otimes \det E$ as in (\ref{L2 metric 3}).

Given $z_0$ in $X$. For the metric $h$ on $O_{P(E^*)}(1)$, as in the formula (\ref{whatever}) from Section \ref{sec pf of thm3}, we have for $\eta\in T^{1,0}_{z_0} X$
\begin{equation}\label{one'}
    -\Theta(\eta,\bar{\eta})\leq -\big( r-\frac{r-M}{4} \big)\Omega(\eta,\bar{\eta}).
\end{equation}

For the metric $g$ on $O_{P(E)}(1)\otimes p^*\det E$, we replace $O_{P(E^*)}(1)\to P(E^*)$ in Section \ref{sec pf of thm3} with $O_{P(E\otimes \det E^*)}(1)\to P(E\otimes \det E^*)$ and use the canonical isomorphism between $O_{P(E\otimes \det E^*)}(1)\to P(E\otimes \det E^*)$ and $O_{P(E)}(1)\otimes p^*\det E \to P(E)$. Then as in Lemma \ref{what ever}, we get for $\eta\in T^{1,0}_{z_0}X$, and $u\in E_{z_0}^*\otimes (\det E_{z_0})$ with $\boldsymbol{H}_1(u,u)=1$, the curvature $\boldsymbol{\Theta}_1$ of $\boldsymbol{H}_1$ satisfies
\begin{equation}\label{two'}
\boldsymbol{H}_1(\boldsymbol{\Theta}_1 u, u)(\eta,\bar{\eta})\leq \big(M+\frac{r-M}{4} \big)\Omega (\eta,\bar{\eta}).
\end{equation}

On the bundle $(E^*\otimes \det E)\otimes \det E^* $, there is a Hermitian metric $\boldsymbol{H}_1\otimes H^*$ with curvature $\boldsymbol{\Theta}_1-\Theta\otimes \Id_{E^*\otimes \det E}$. As a result of (\ref{one'}) and (\ref{two'}), we deduce that for $\eta\in T^{1,0}_{z_0}X$, and $u\in E_{z_0}^*\otimes (\det E_{z_0})$ with $\boldsymbol{H}_1(u,u)=1$, 
\begin{equation}
\boldsymbol{H}_1(\boldsymbol{\Theta}_1 u, u)(\eta,\bar{\eta})    -\Theta(\eta,\bar{\eta})\leq \big(M+\frac{r-M}{4} \big)\Omega (\eta,\bar{\eta})
-\big( r-\frac{r-M}{4} \big)\Omega(\eta,\bar{\eta});
\end{equation}
the term on the right is negative. So the Hermitian metric $\boldsymbol{H}_1\otimes H^*$ is Griffiths negative at $z_0$ an arbitrary point. Hence $\boldsymbol{H}_1\otimes H^*$ is Griffiths negative on $X$, and the bundle $E$ is Griffiths positive.

\section{Examples}\label{example}
\begin{example}
We provide here an example where the assumptions in Theorems \ref{thm 1},\ref{thm 2},\ref{thm 3}, and \ref{thm 4} are satisfied. Let $L$ be a line bundle with a metric $H$ whose curvature $\Theta>0$. Let $E=L^9 \bigoplus L^8\bigoplus  L^7$ a vector bundle of rank $r=3$. The induced metric $(H^*)^9 \bigoplus (H^*)^8 \bigoplus (H^*)^7 $ on the dual bundle $E^*$ has curvature $\Theta(E^*)=(-9\Theta) \bigoplus (-8\Theta) \bigoplus (-7\Theta)$ which is Griffiths negative, so the corresponding metric $h$ on $O_{P(E^*)}(1)$ is positively curved. According to (\ref{Herm}), we see
\begin{align*}
    -\theta(h)=-q^*\Theta(E^*)|_{O_{P(E^*)}(-1)},
\end{align*}
hence we have 
\begin{equation}\label{ineq}
 7q^*\Theta \leq -\theta(h)\leq 9q^* \Theta.   
\end{equation}
 For all four theorems, we will use this metric $h$ on $O_{P(E^*)}(1)$ and take $\Omega$ to be $7\Theta$. Therefore $q^*\Omega \leq -\theta(h)$ always holds. The choice of $g$ will be different from case to case.

For Theorem \ref{thm 1}, we choose $g$ to be $h$, hence by (\ref{ineq}) and $\Omega=7\Theta$ we get 
\begin{equation}
    q^*\Omega \leq -\theta(h)=-\theta(g)\leq \frac{9}{7} q^* \Omega.
\end{equation}
To fulfill the assumption of Theorem \ref{thm 1}, we can choose $M=9/7$ which is in the interval $[1,3)$.

For Theorem \ref{thm 2}, since $E\otimes \det E^*= (L^*)^{15} \bigoplus (L^*)^{16}\bigoplus  (L^*)^{17}$ has induced curvature $(-15\Theta) \bigoplus (-16\Theta) \bigoplus (-17\Theta)$ which is Griffiths negative, the corresponding metric $g$ on $O_{P(E)}(1)\otimes p^*\det E$ is positively curved and satisfies 
\begin{equation}
    15p^*\Theta \leq -\theta(g) \leq 17p^*\Theta.
\end{equation}
Together with (\ref{ineq}) and $\Omega=7\Theta$, we have 
\begin{equation}
     \frac{17}{7} p^*\Omega \geq -\theta(g)   \text{ and } q^*\Omega \leq -\theta(h).
\end{equation}
We can choose $M=17/7$ which is in $[1,3)$.

For Theorem \ref{thm 3},
notice that $h$ is induced from $(H^*)^9 \bigoplus (H^*)^8 \bigoplus (H^*)^7 $ on $E^*$, so if we use  $(H^*)^9 \bigoplus (H^*)^8 \bigoplus (H^*)^7 $ for the Hermitian metric $G$, then the corresponding $g$ is actually $h$. Since $\Theta(\det G)=-24 \Theta$, by using (\ref{ineq}) we have 
\begin{equation}
    -(r+1)\theta(g)+q^*\Theta(\det G)=-4\theta(h)-24q^*\Theta\leq 12 q^* \Theta=\frac{12}{7} q^* \Omega.
\end{equation}
We choose $M=12/7$ which is in $[1,3)$.

Finally, for Theorem \ref{thm 4}. On $E\otimes \det E^*= (L^*)^{15} \bigoplus (L^*)^{16}\bigoplus  (L^*)^{17}$, we will use the metric $(H^*)^{15} \bigoplus (H^*)^{16} \bigoplus (H^*)^{17} $ for $G$, so $\Theta(\det G)=-48\Theta$. Moreover, 
the corresponding metric $g$ on $O_{P(E)}(1)\otimes p^*\det E$ satisfies 
\begin{equation}
    15p^*\Theta \leq -\theta(g) \leq 17p^*\Theta,
\end{equation}
so we get 
\begin{equation}
    -(r+1)\theta(g)+p^*\Theta(\det G)\leq 20 p^* \Theta= \frac{20}{7}p^* \Omega.
\end{equation}
We choose $M=20/7$ which is in $[1,3)$.
\end{example}

\begin{example}
Let $X$ be a compact Riemann surface with a Hermitian metric $\omega$. Let $E$ be an $\omega$-semistable ample vector bundle of rank $r$ over $X$. The assumptions in Theorems \ref{thm 1},\ref{thm 2},\ref{thm 3}, and \ref{thm 4} are all satisfied in this case. We will explain for only Theorems \ref{thm 2} and \ref{thm 4}. Theorems \ref{thm 1} and \ref{thm 3} can be verified similarly. By \cite[Theorem 1.7, Remark 1.8, and Theorem 1.11]{li2021hermitian}, there exists a constant $c>0$, such that for any $\delta>0$, there exists a Hermitian metric $H_\delta$ on $E$ satisfying 
\begin{equation}\label{HYM}
    (c-\delta)\Id_E \leq \sqrt{-1}\Lambda_\omega \Theta(H_\delta) \leq (c+\delta)\Id_E,
\end{equation}
where $\Lambda_\omega $ is the contraction with respect to $\omega$. Since $X$ is a Riemann surface, $\Lambda_\omega $ locally is multiplication by a positive function.

For Theorem \ref{thm 2}, we choose $\delta =c/5r$. The Hermitian metric $H^*_\delta$ on $E^*$ induces a metric $h$ on $O_{P(E^*)}(1)$. Due to (\ref{Herm}), we see
\begin{align}\label{i guess}
    -\theta(h)=-q^*\Theta(H^*_\delta)|_{O_{P(E^*)}(-1)};
\end{align}
combining with (\ref{HYM}), we have 
\begin{equation}\label{out of label}
 (c-\delta) q^*\omega \leq -\theta(h)\leq ( c+\delta )q^* \omega.
\end{equation}
The Hermitian metric $H_\delta\otimes \det H^*_\delta$ on $E\otimes \det E^*$ induces a metric $g$ on $O_{P(E)}(1)\otimes p^*\det E$. Similar to (\ref{out of label}), we have 
\begin{equation}\label{g}
    -\theta(g)\leq [-(c-\delta)+r(c+\delta)]p^*\omega.
\end{equation}
If we choose $\Omega=(c-\delta)\omega$ and $M=r-1/2$, then $[-(c-\delta)+r(c+\delta)]p^*\omega\leq Mp^* \Omega$. As a result, we achieve the assumption in Theorem \ref{thm 2}: $q^*\Omega \leq -\theta(h)$ and $ -\theta(g)\leq Mp^* \Omega$.

For Theorem \ref{thm 4}, we choose $\delta=c/9r$. We still have (\ref{out of label}). The Hermitian metric $G$ on $E\otimes \det E^*$ is taken to be $H_\delta\otimes \det H^*_\delta$, so we get 
\begin{align*}
    &-(r+1)\theta(g)+p^*\Theta(\det G)\\
    =&
    -(r+1)\big[p^*\Theta(H_\delta)|_{O_{P(E)}(-1)}-p^*\Theta(\det H_\delta)\big]-(r-1)p^*\Theta(\det H_\delta)
        \\
        \leq& \big[-(r+1)(c-\delta)+2r(c+\delta)\big]p^*\omega.
\end{align*}
If we choose $\Omega=(c-\delta)\omega$ and $M=r-1/2$, then $[-(r+1)(c-\delta)+2r(c+\delta)]p^*\omega\leq M p^*\Omega$. So the assumption of Theorem \ref{thm 4} is satisfied. 

In light of \cite[Theorem 1.7]{li2021hermitian}, it is possible to modify our theorems so that semistability is not needed in this example.

\end{example}

\bibliographystyle{amsalpha}
\bibliography{Dominion}

\textsc{Institute of Mathematics, Academia Sinica, Taipei, Taiwan}

\texttt{\textbf{krwu@gate.sinica.edu.tw}}

\end{document}